\documentclass[11pt]{amsart}
\usepackage{amsmath}
\usepackage{amsfonts}
\usepackage{amsthm}
\usepackage{amssymb}
\usepackage{amscd}
\usepackage[all]{xy}
\usepackage{enumerate}
\usepackage{color}
\usepackage{bm}
\usepackage{amstext}



\theoremstyle{plain}
\newtheorem{thm}{Theorem}[subsection]

\newtheorem{prop}[thm]{Proposition}

\newtheorem{lem}[thm]{Lemma}

\theoremstyle{definition}

\newtheorem{conj}[thm]{Conjecture}
\newtheorem{prob}[thm]{Problem}

\newtheorem*{ackn}{Acknowledgements}

\newtheorem{rem}[thm]{Remark}

\newcommand{\sH}{\mathcal{H}}

\newcommand{\sN}{\mathcal{N}}
\newcommand{\sM}{\mathcal{M}}
\newcommand{\sO}{\mathcal{O}}

\newcommand{\sQ}{\mathcal{Q}}

\newcommand{\sU}{\mathcal{U}}

\newcommand{\mG}{\mathbb{G}}

\newcommand{\mP}{\mathbb{P}}

\newcommand{\Q}{Q}

\newcommand\K{\mathbb{K}}
\newcommand\Z{\mathbb{Z}}

\newcommand{\A}{\mathbb{A}}

\numberwithin{equation}{section}

\author{Pietro Corvaja Francesco Zucconi}

\address{\newline D.M.I.F. \\
Univerit\`a degli studi di Udine\\
Udine, 33100, Italy.\newline
\texttt{pietro.corvaja@uniud.it\newline
francesco.zucconi@uniud.it}}

\begin{document}

\begin{center}
\textbf{\Large
On integral points of some Fano Threefolds and their Hilbert schemes of lines and conics\\}
\par
\end{center}
{\Large \par}

$\;$

\begin{center}
Pietro Corvaja, Francesco Zucconi 
\par\end{center}

\vspace{5pt}

$\;$

Abstract. {\it{We prove some density results for integral points on affine open sets of Fano threefolds. For instance,  let $X^o=\mathbb P^3\setminus D$ where $D$ is the union of two quadrics such that their intersection contains a smooth conic, or the union of a smooth quadric surface and two planes, or the union of a smooth cubic surface $V$ and a plane $\Pi$ such that the intersection $V\cap\Pi$ contains a line. In all these cases we show that the set of integral points of   
$X^o$ is potentially dense. We apply the above results to prove that integral points are potentially dense in some log-Fano or in some log-Calabi-Yau threefold.}}

\vspace{0.5cm}




\section{Introduction} 

\subsection{General overview}
Let $ {X}\subset\mathbb P^n$ be a projective algebraic variety over the number field $\mathbb K$ and let $D$ be a closed algebraic subset. Set $X^o:= {X}\setminus D$, which is a quasi projective (possibly affine) algebraic variety.  A rational point $P\in X^o(\mathbb K)$ is said to be integral with respect to $D$ if for no finite place $\nu$ of $\mathbb K$ the point $P$ reduces to $D$ modulo $\nu$. More generally, if $S$ is a finite set of places containing the archimedean ones, a point is said to be $S$-integral with respect to $D$ if the same condition holds for every place $\nu$ not contained in $S$.  Whenever $X^o$ is affine, say embedded into $\A^n$, this set can be identified with the set of points of $X$ whose coordinates belong to the ring of $S$-integers $\sO_S\subset \K$. Hence we shall  denote by $X^o(\sO_S)$ the set of $S$-integral points of $X$. Note that whenever $D$ is empty, that is $X^o=X$, the set  $X^o(\sO_S)$ coincides with the set $X(\K)$ of rational points.
\smallskip

\subsection{Vojta's conjecture}  
A celebrated conjecture of Vojta,  see: \cite[Main Conjecture 3.4.3 and Proposition 4.1.2]{vojta},  predicts that if ${X}$ is smooth and $D$ is a reduced effective divisor, with at most normal crossing singularities and such that $K_{X}+D$ is   big, where $K_{{X}}$ denotes a canonical divisor of ${X}$, then for every ring of $S$-integers $\sO_S\subset\mathbb K$, the $S$-integral points on $X^o=X\setminus D$ are not Zariski-dense.  Note that the bigness assumption for $K_{{X}}+D$ turns out to be independent of the smooth compactification of $X^o$, if $D$ has at most normal crossing singularities.

\subsubsection{The $1$-dimensional case of Vojta's conjecture} Vojta's Conjecture is completely settled in the case of curves, in view of theorems of Siegel and Faltings; moreover, in this case, a converse holds true: if for a pair $({X},D)$, where ${X}$ is a smooth complete curve and $D\subset X$ is a reduced divisor (a finite set), such that $K_X+D$ in  not big then there exists a ring of $S$-integers (possibly in a finite extension of the field of definition $\K$) such that $X^o(\sO_S)$ is Zariski-dense (i.e. an infinite set). Actually such pairs are necessarily of one of the following four classes: $X=\mP^1$ and $|D|=0,1,2$ or $X$ a genus one curve and $D=\emptyset$. In each case $X^o$ is a homogenous space under the action of an algebraic group; see: c.f. \cite{C-libretto-verde}.

\subsubsection{The $2$-dimensional case of Vojta's conjecture} Already in dimension two, the conjecture is widely open. For an open set of the projective plane, Vojta's conjecture predicts degeneracy of integral points whenever the divisor at infinity $D$ has degree $\geq 4$ (and at most normal crossing singularities). An application of the Schmidt's Subspace Theorem in Diophantine approximation leads to this degeneracy result whenever $D$ has at least four component. The same conclusion holds for the complement of four  ample divisors on any algebraic surface (proved by A. Levin \cite{levin}, based on the work of U. Zannier and the first author \cite{CZ-Annals}, ultimately relying on the Subspace Theorem). 
Several other special cases are proved but  already the case of the complement of a plane quartic with $\leq 3$ components is still open. 

\subsubsection{Vojta's conjecture and semi-abelian varieties} A general result, based on the work of Faltings and Vojta on the distribution of integral points on semi-abelian varieties, proves the degeneracy of integral points on a quasi-projective variety $X^o $ whenever its generalized Albanese variety has dimension strictly larger than $\dim X$. Recall that the generalized Albanese variety of $X ^o$ is a semi-abelian variety $A$, provided  with a morphism $X ^o \to A$ factoring any other morphism from $X $ to any semi-abelian variety.

\subsubsection{Campana' s conjecture and a converse for Vojta's conjecture} A conjecture of Campana provides a converse for Vojta's Conjecture: whenever a variety is \lq special\rq, in a sense which we shall not discuss here, its integral (or rational) points should be potentially dense, which means that they become Zariski-dense after a  finite extension of the ground field (possibly followed by an enlargment of the finite set $S$).

In the above mentioned case of the complement of a normal crossing divisor $D$ in $\mP^2$, a variety is special if and only if $\deg D\leq 3$. In these cases, potentially density of integral points is known.

For a survey on this topic, especially for the case of dimension two, we refer to \cite{C-libretto-verde}. 

\subsection{The $3$-dimensional projective space}
In this paper we study the problem of potential density of integral points for certain open subsets of some rational Fano $3$-folds.

\subsubsection{Log Fano and of log Calabi-Yau varieties} We say that an affine variety $X^o$ is {\it log Fano} if it can be obtained as $X^o= {X}\setminus D$, where $-K_{ {X}}-D$ big. We say that $X^o$ is {\it log Calabi-Yau} if $X^o= {X}\setminus D$ where $D$ lies in the anti-canonical class (so 
$K_{ {X}}+D\sim 0$). In both cases, the integral points are conjectured to be potentially dense.
\smallskip

In dimension two, for every log-Fano (or log  del Pezzo) surface and every affine log Calabi-Yau (or log-$K3$) surface potential density is well studied. Clearly, the crucial case lies in the log-$K3$ case, which is due to Hassett-Tschinkel \cite{Hassett-Tschinkel}, after  
the fundamental work by F. Beukers \cite{Beukers}, (to be precise, Hassett-Tschinkel proved potential density for a surface of the form $ X\setminus D$, where $X$ is a del Pezzo surface and $D$ is a smooth curve in the anti-canonical class. As it usually happens, however, the case of a singular - possibly reducible - divisor, still with normal crossing singularities, is easier. The details have been provided by S. Coccia).

\subsubsection{Our results on the $3$-dimensional projective space}  We stress that after Siegel's theorem on integral points on curves (see e.g. \cite[Theorem 3.3.1]{C-libretto-verde} or \cite[Theorem 3.9]{CZ-libro}), the only affine curves admitting infinitely many integral points are those of genus zero with one or two points at infinity (namely the log-Fano and log-Calabi-Yau in dimension one!). Now, a way to prove density of integral points on higher dimensional affine varieties consists in covering the variety (or a Zariski-open subset of it) with rational curves with just one or two points at infinity. These curves might possess infinitely many integral points, giving rise to a Zariski-dense set on the variety; see Lemma \ref{lemma-beukers}, which we have called Beukers Lemma. 

Hence, it might be interesting to describe the Hilbert scheme of lines or of conics on the given Fano variety, where by lines (resp. conics) we mean those curves which are sent to lines (resp. conics) under a suitable embedding. If the divisor $D$ at infinity in such an embedding consists on a hyperplane section (resp. the union of two hyperplane sections), then the lines and the conics (resp. the lines only) are expected to possess infinitely many integral points. This last fact however holds only after a suitable extension of the ground field and/or of the set of places $S$. 

In fact in paragraph \ref{S.fully-integral} we shall introduce the new notion of {\it fully integral} rational curve; these rational curves form a family of curves each possessing infinitely many integral points over any ring of $S$-integers with infinitely many units. In some cases, the potential density of integral points on Fano threefolds follows from the existence of a sufficiently large family of fully integral curves. 
\smallskip

Let us come back to Fano threefolds. In the particular case where $X=\mathbb P^3$, $D$ is a surface (with normal crossing singularities) and $X^o=\mP^3\setminus D$, it holds that $X^o$ is log-Fano whenever $\deg D\leq 3$ and log-Calabi-Yau for $\deg D=4$. 
The case of log-Fano ($\deg D\leq 3$) is probably well-known; however, in some very simple cases we know no reference for a proof; for instance our Theorem \ref{cubicandplanewithacommonline} settles the case for the complement of a smooth cubic surface in $\mP^3$ in a stronger form.

The crucial log-Calabi-Yau case, $\deg D=4$, is still open. We believe that when $D$ is smooth (in particular irreducible) it consists in a deep problem; even its analogue in complex-analysis, namely the existence of a Zariski-dense entire curve in $\mP^3$ omitting $D$, is still unknown. To have an idea of the limitations of the straight lines technique mentioned above, consider the basic case where $D$ is a smooth quartic surface. The integral points on $X^o=\mathbb P^3\setminus D$ correspond also (via the Chevalley-Weil Theorem - see \cite{CTZ} for a modern treatment) to the integral points on the double solid $X'$ branched over $D$ with respect to a divisor $D'$ (the pre-image of $D$). The lines on $X'$ correspond to the bitangents (in $\mP^3$) to $D$; however, the attempt to prove the potential density of integral points on $X'$ or on ${X'}^o:=X'\setminus D'$ by constructing integral points on such lines is  doomed to failure, since there exist only finitely many such lines which are defined over a given number field (this is  Theorem B in our recent work \cite{CoZucc-secondo}). 

Even the case where $D$ is the union of two quadrics $Q_1$, $Q_2$ is not fully solved. In this context we have a natural notion of bitangent line to $Q_1\cup Q_2$. The surface of bitangents turns out to be the union of a Kummer surface and a ruled surface with elliptic base; in both surfaces, rational points are proved to be Zariski-dense. However, as we shall  explain in Theorem \ref{few-fully-integral}, the  fully integral curves  are not enough to generate a Zariski-dense set of integral points: this is proof of the fact that the method of using straight lines has strong limitations. Nevertheless, when the two quadrics intersect in the union of two conics, then the construction of integral points based on the family of fully integral curves does work and leads to:
\medskip

\noindent
{\bf{Theorem [A]}}{\it{
Let $X^0$ be the complement of $\mathbb P^3$ by the union of two quadrics such that their intersection contains a smooth conic. Then the integral points of $X^o$ are potentially dense.}}

\medskip

We provide two proofs of the above theorem, the first one exploiting a two-dimensional family of conics on $X^o$, which turn out to be fully integral. 
We think quite interesting to compare the arithmetical proof of the above Theorem [A], see the proof of Theorem \ref{two-quadrics-two concis}, which is grounded on the Beukers' result recalled in Lemma \ref{lemma-beukers}, with a proof based on the birational geometry, as the one of Theorem \ref{complement by two quadrics special case}.

For the case where $D$ has three components (necessarily a quadric and two planes) the conjecture is settled:
\medskip

\noindent
{\bf{Theorem [B]}}{\it{
Let $X^o$ be the complement of $\mathbb P^3$ by the union of a smooth quadric surface $Q$ and two planes $\Pi_1$, $\Pi_2$ such that the intersection $Q\cap\Pi_1\cap\Pi_2$ is proper (i.e. it consists of two points).
Then the integral points on $X^o$ are potentially dense.}}
\medskip

See: Theorem \ref{quadricaeduepiani}.

In Theorem \ref{cubicandplanewithacommonline} we consider a special log-Calaby-Yau case:

\medskip

\noindent
{\bf{Theorem [C]}}{\it{
Let $X^o$ be the complement of $\mathbb P^3$ by the union of a smooth cubic surface $V$ and a plane $\Pi$, such that the intersection $V\cap\Pi$ contains a line.Then the integral points on $X^o$ are potentially dense.}}
\medskip

Actually the general case of the complement of a smooth cubic $V$ and a plane $\Pi$ is still unknown even if in Conjecture \ref{complement by plane and cubic} we favor an affirmative answer.

\subsection{Other rational Fanos $3$-folds}
Notwithstanding the above results are special, they also provide results for open sets of other rational Fanos. In this work we present some cases which should shed some light on the full topic. For example, in \cite[Theorem 1.1. Table 1]{BL} are collected those Fano threefolds $A$ obtained by blowing-up $\mathbb P^3$ along a curve $\Gamma$ and such that $-K_A$ is big and nef. A systematic study on the integral points for all the Fanos obtainable in this way is beyond the aims of this work. Actually those ones where $-K_A$ is ample and the use of Proposition \ref{quadricaeduepiani} is quite straightforward are studied in this work. Indeed it seems us that, at least in the case where $-K_A$ is ample, the images of the various exceptional divisors have a degree too big to apply successfully Proposition \ref{quadricaeduepiani}, except in those cases treated in this paper. We intend to study some of the remaining cases, that is when $-K_A$ is big and nef, that is the case where $A$ is a weak Fano, in a future project.

\subsubsection{The quadric threefold}
The quadric threefold $Q_3\subset\mathbb P^4$ has index $3$. We have:

\medskip

\noindent
{\bf{Theorem [D]}}{\it{The integral points in the complement of the smooth quadric threefold by three hyperplane sections or by a quadric section are potentially dense.}}
\medskip

See Proposition \ref{treiperpiani} and Proposition \ref{dueiperpiani}. Actually by projection from a point of the smooth quadric the proof is reduced to the case of the complement of $\mathbb P^3$ by four projective planes which is $\mathbb G_M^3$ and in this case potential density is known. Strangely enough the proof of the log-Fano case where $D$ is a smooth quadric section is harder; see: Proposition \ref{dueiperpiani}. The projection from a point $P\in D$ reduces to the case of the complement of $\mathbb P^3$ by the union of a smooth cubic $V$ and a plane $\Pi$ which have a line in common. As we mentioned above we prove in Theorem [C] that the corresponding set of integral points is potentially dense. 

\subsubsection{Rational Fanos with Picard number $1$ and index $1$ or $2$}
Basically, we shall be interested in the following situation: given a rational Fano threefold of index $1$ and with Picard number $1$ embedded in a projective space via its anti-canonical line bundle, study the integral points in the complement of one hyperplane section. Whenever the Fano threefold has index $2$, we shall consider both the complement of one hyperplane section (a log-Fano variety) and the complement of two hyperplane sections (log-Calabi-Yau); needless to say, the proof of the potential density of integral points in the last case is usually more difficult.

If $X=Q_1\cap Q_2\subset\mathbb P^5$  is the smooth quadric complex we fully solve the log-Fano case and the log-Calabi-Yau one only in special cases.
\medskip

\noindent
{\bf{Theorem [E]}}{\it{The set of integral points of the complement of a smooth quadric complex in $\mathbb P^5$ by an hyperplane section or by two hyperplane sections which share a common line is potentially dense.}}
\medskip

\noindent
See Proposition \ref{quadriccomplexCY} and Proposition \ref{quadcompline}.

If $X$  is the del Pezzo threefold again we can show the log-Calabi-Yau case in the special case where the two hyperplane sections have a line in common. \medskip

\noindent
{\bf{Theorem [F]}}{\it{The set of integral points of the complement on the del Pezzo threefold in $\mathbb P^6$ by two hyperplane sections which have a smooth conic in common is potentially dense.}}
\medskip

\noindent
See Theorem \ref{delpezzoconica}. Actually by projection from the common conic the proof is reduced to Theorem [B]. Clearly Theorem [F] implies the log-Fano case and again we provide also an arithmetical proof of this last case, see sub-section \ref{casodidelpezzomenouno}. This arithmetical proof relies on the fact that the Hilbert scheme of lines on the del Pezzo $3$-fold is $\mathbb P^2$ and that the set of integral points of the complement of $\mathbb P^2$ by the algebraic closed set formed by the union of three lines in general position and a finite set of points is potentially dense; see: Lemma \ref{lemma-punture}.
Finally we can show that the subset given by the integral points of the complement of a cubic threefold with an ordinary double point by two hyperplane sections passing through the singular point is potentially dense; see: Proposition \ref{cubic threefold}.

\subsection{Integral points and the Hilbert scheme of lines and conics}

As we saw above, in some cases we can show the density of integral points on higher dimensional affine varieties if there are many rational curves with just one or two points at infinity. Indeed the study of the Hilbert scheme of lines and of conics is a natural step to tackle with our problems on the density of integral points. 
The first section of the paper concerns a result, as interesting as easy to be proved, on the arithmetic of Hilbert scheme of lines and of conics for almost all deformation classes of smooth Fano threefolds.  At a first reading, the reader can skip to read Section 1 without prejudice to his understanding of the rest. Indeed we show by a case by case check, that, letting aside very special deformation classes, the subset of rational points of the Hilbert schemes of lines or of conics is degenerate except in the obvious cases. In Remark \ref{moltoimportante} we propose the new notion of $\mathbb K$-standard Fano. This notion is crucial to adapt to the arithmetical study of Fanos the property of a Fano $X$ to be generic in its deformation class when one works over $\mathbb C$. For the Hilbert scheme of lines $\Sigma(X)$ of a $\mathbb K$-standard Fano $X$ of index $1$ or $2$ and Picard number $\rho=1$, we show in Proposition \ref{aritmetica1} that the set of its $\mathbb K$-rational lines is degenerate except if $X$ is a quadric complex or the del Pezzo $3$-fold of degree $5$ and in these last two cases this set is dense. Instead for the Hilbert scheme of conics, we show in Proposition \ref{aritmetica2} that the set of its $\mathbb K$-rational conics is degenerate except if $X$ is the del Pezzo $3$-fold of degree $5$ or $X$ has genus $10$ or $12$ where density holds. As a by-product of our review is a new problem, to our knowledge never raised in the literature, about the size of the image of the Abel-Jacobi map, which we point out the reader in Problem \ref{problemaaperto}.

\section{Hilbert Schemes of Lines and Conics}
This section is taken from \cite[Sect. 4.1]{IP}, \cite[Sect. 2]{KPS} and we have followed the presentation of \cite[Sec. 5]{CZ} and \cite[Subsection 3.1]{IM1}.

\subsection{Fano Threefolds}
A smooth Fano threefold $X$ is a smooth threefold with ample anti-canonical divisor. To classify Fano threefolds means that given the values of certain biregular invariants of $X$ then $X$ belongs to a class of smooth varieties which can be described in a very explicit geometrical way. 

The Fano-Iskovskikh classification relies on the following $5$ biregular invariants of $X$:
$$
\rho(X),\, g(X), \,\iota(X),\, d(X),\, m_0(X)
$$
\noindent
where $\rho(X)$ is the {\it{Picard rank}} of $X$. By c.f. \cite[Proposition 2.1.2]{IP} $\rm{Pic}(X)$ is the torsion free abelian group of rank $\rho$, $g(X)$ is {\it{the genus}} and it holds:
$$
g(X) = \frac{(-K_X )^3}{2} + 1 = {\rm{dim}}_\mathbb C \mid - K_X \mid -1 \geq 2
$$ 
\noindent
The number $\iota(X)$ is called {\it{the index}}. It is the maximal natural number such that for the class $[- K_X]\in \rm{Pic}(X)$ it holds that $[- K_X]=\iota(X)[H]$ where $H$ is an ample divisor; $d(X)=H^3$ is {\it{the degree of}} $X$ and finally $m_0(X)$ is {\it{the Matsusaka constant}}, that is the least integer such that $m_0(X)[H]$ is the class of a very ample divisor.
Since $g(X) = {\frac{\iota(X)^3d(X)}{2}} + 1$ then it holds that $d(X) =2g(X) -2$ if and only if $\iota(X) = 1$. We remark that by Mori-Mukai's classification of smooth Fano threefolds of rank $\rho\geq 2$ there are 87 families of such Fano threefolds, \cite{MM1}, \cite{MM2}, \cite{MM3}. As far as we know there is not an available classification of their Hilbert schemes of lines. From now on we consider Fano threefolds with Picard rank equal to $1$. 
\subsection{Smooth Fano threefolds with $\rho(X)=1$} These threefolds are distributed into $17$ deformation families. They can be grouped according to the index $\iota(X)$ which is a positive natural number less than $4$. If $\iota(X)=4$ then $X=\mP^3$ and if $\iota(X)=3$ then $X$ is the smooth quadric in $\mP^4$, see c.f. \cite[Theorem 3.1.14, and Proposition 3.1.15]{IP}. Below we consider the cases where $\iota(X)=1$ or $\iota(X)=2$. 
\subsubsection{Smooth Fano threefolds with $\rho(X)=1$ and $\iota(X)=2$} 
The $5$ families of Fano threefolds with $\iota(X)=2$ and $\rho(X)=1$  are classified according to the anticanonical degree 
$(-K_{X} ) ^{3}=8d$ where $d=1, 2, 3, 4, 5$ as follows: 

\begin{enumerate}
\item $d=1$, $X$ is a hypersurface in $\mP(1, 1, 1, 2, 3)$ of degree $6$;
\item $d=2$, $X$ is the double cover of $\mP^3$ branched in a quartic surface;
\item $d=3$, $X$ is a cubic hypersurface in $\mP^4$;
\item $d=4$, $X$ is the complete intersection of two quadrics in $\mP^5$;
\item $d=5$, $X=B_5$ is a transversal section of $\mathbb G(2, 5)\subset\mP^9$ by a linear subspace of codimension $3$.
\end{enumerate}
We stress that  by a result of Iskovskikh in the case $d=5$, $B_5$ is unique (up to biregular morphisms); see: c.f. \cite[Corollary 3.4.2.]{IP}.

\subsubsection{Smooth Fano threefolds with $\rho(X)=1$ and $\iota(X)=1$} 
The $10$ families of Fano threefolds with $\iota(X)=\rho(X)=1$  are classified according to the genus $g$ as follows:

\begin{enumerate}
\item $g = 2$, $X$ is a double cover of $\mP^3$ branched over a smooth sextic surface;
\item $g = 3$, $X\subset\mP^4$ is a smooth quartic threefold, or a double cover of a smooth quadric in $\mP^4$ branched in an intersection with a quartic;
\item $g = 4$, $X\subset\mP^5$ is a complete intersection of a quadric and a cubic;
\item $g = 5$, $X\subset\mP^6$ is a complete intersection of three quadrics;
\item $g = 6$, $X$ is section of $\mathbb G (2,5)\subset\mP^9$ by a linear subspace of codimension $2$ and a quadric, or $X$ is a double cover of the Fano threefold $B_5$ branched in an anticanonical divisor;
\item $g = 7$, $X$ is section of a connected component of the orthogonal Lagrangian Grassmannian $\mathbb O\mathbb G_{+}(5,10) \subset \mP^{15}$ by a linear subspace of codimension $7$;
\item $g = 8$, $X$ is a section of $\mathbb G (2,6)\subset\mP^{14}$ by a linear subspace of codimension $5$;
\item $g = 9$, $X$ is a section of the symplectic Lagrangian Grassmannian $\mathbb L\mathbb G(3, 6)\subset\mP^{13}$ by a linear subspace of codimension $3$;
\item $g = 10$, $X$ is a section of the homogeneous space $G_2/P\subset\mP^{13}$ by a linear subspace of codimension 2;
\item g = 12, $X$ is the zero locus of three sections of the rank-$3$ vector bundle $\bigwedge^2\sU^\vee$ where $\sU$ is the universal subbundle on $\mathbb G(3, 7)$.
\end{enumerate}
As often remarked by several authors, the two cases occurring for $\rho(X)=\iota(X)=1$ and $g=6$ belong to the same deformation family and the same occurs for the case $\rho(X)=\iota(X)=1$ and $g=3$.

\subsection{Restrictions}
By \cite[Lemma 2.1.1]{KPS}, see also \cite[Remark 2.1.2]{KPS} in order to avoid pathologies in the shape of the Hilbert schemes, in this work we will restrict our considerations mainly to those deformation families where $2H$ is very ample in the case of lines and to the case where $H$ is very ample in the case of conics. In other words, for the Hilbert schemes of lines we will not consider the case where $\rho(X)=1$, $\iota(X)=2$ and $X$ is a hypersurface in $\mP(1, 1, 1, 2, 3)$ of degree $6$ and $\iota(X)=2$) nor the case where $\rho(X)=\iota(X)=1$ and $X$ is a double cover of $\mP^3$ branched over a smooth sextic surface. 

For the Hilbert scheme of conics, we will exclude both the cases where $\rho(X)=1$ and $\iota(X)=2$ and $X$ is a hypersurface in $\mP(1, 1, 1, 2, 3)$ of degree $6$ or $X$ is the double cover of $\mP^3$ branched in a quartic surface than the cases where $\rho(X)=\iota(X)=1$ and $X$ is a double cover of $\mP^3$ branched over a smooth sextic surface or $X$ is a double cover of a smooth quadric in $\mP^4$ branched in an intersection with a quartic. 

\subsection{Hilbert Schemes of Lines where $\rho(X)=1$ and $\iota(X)=2$}
We recall that if $\iota(X)=4$ then $X=\mathbb P^3$ and the Hilbert scheme of lines is the Grasmannian ${\mathbb G}(2,4)$ which is realised as the Klein quadric inside $\mathbb P^5$, while if $\iota(X)=3$ then $X=Q_3$ is a smooth quadric in $\mathbb P^4$ and the Hilbert scheme of lines is $\mathbb P^3$; see c.f. \cite[Section 6]{Tan}.

We also recall that if $\rho(X)=1$, $\iota(X)=2$ and $d=1$, the first multiple of $H$ to be very ample is $3H$: this gives rise to an extra component in moduli. Hilbert scheme ${\rm{Hilb}}(X, t + 1)$ has two irreducible components M1 and M2, parametrizing smooth lines and genus 1 curves union a point, respectively, see: \cite[Theorem 3.7]{PR}. We do not study this case.

\medskip

For the remaining cases, $d(X)=2,3,4,5$, it holds that $2H$ is very ample. We can split our description in two subcases: $d(X)=2$, that is the case where $X$ is a quartic double solid, and the case where $d(X)\geq 3$.

\begin{prop}\label{bitangents} If $X$ is a quartic double solid branched over a quartic surface which contains no lines then $\Sigma(X)$ is a smooth surface of general type with the following invariants: $q(\Sigma(X)) = 10$, $p_g(\Sigma(X)) = 101$, $h^1(\Sigma(X),\Omega^1_{\Sigma (X) }) = 220$, $c_2(\Sigma(X) ) = 384$. 
\end{prop}
\begin{proof}
See \cite[Cohomological study pp. 41-45]{W}. See also \cite{Tih2}, and \cite{CoZucc-primo}.
\end{proof}

If $d\geq 3$ we have:

\begin{prop}\label{gradomaggioretre}
Let $X$ be a smooth Fano threefold with $\rho(X)=1$, $\iota(X)=2$ and $d(X) \geq 3$. Then the Fano scheme of lines $\sigma(X)$ is a smooth irreducible surface. In particular,
\begin{enumerate}
\item if $d = 3$ then $\Sigma(X)$ is a minimal surface of general type with  $q(\Sigma (X))=5$, $p_g(\Sigma(X))=10$, and $K_{\Sigma (X)}^2 =45$;
\item if $d = 4$ then $\Sigma(X)$ is an abelian surface;
\item if $d=5$ then $\Sigma(X)=\mP^2$.
\end{enumerate}
\end{prop}
\begin{proof}
See c.f. \cite[Thm. 1.1.1]{KPS}.
\end{proof}

\subsection{Hilbert Schemes of Lines where $\rho(X)=1$ and $\iota(X)=1$}
If $g\geq 4$ or $g=3$ and $X$ is a quartic threefold then the anticanonical divisor is very ample, and the linear system $\mid -K_X\mid$ defines an embedding of X onto a projectively normal threefold  of degree $2g-2$ see c.f. \cite{CZ}. By \cite[Lemma 2.2.3]{KPS} we know that every irreducible component of $\Sigma(X)$ is one dimensional and that if $X$ is general in its deformation class then $\Sigma(X)$ is a smooth curve of positive genus see: \cite[Theorem 4.2.7]{IP}.

%
If $g=2$ or $g = 3$ and $\mid K_X\mid$ defines a double cover over a smooth quadric threefold $Q\subset\mP^4$ ramified along a smooth surface $S \subset Q$ of degree $8$ then $2H$ is not very ample and we do not study here these cases.
\medskip

The following remark is crucial.
\begin{rem}\label{moltoimportante} In many cases $\Sigma(X)$ can have more than an irreducible component and in some cases everywhere nonreduced components occur. On the other hand by the very construction of Fano $3$-folds and of their Hilbert schemes it holds that the description of $\Sigma(X)$, and as well of $S(X)$, obtained for the general $X$ in the corresponding deformation family works perfectly for a dense set in this family given by Fanos defined over $\mathbb Q$ or on a number field $\mathbb K$. Hence from now on, when we will consider Fanos defined over $\mathbb K$ we will always refer to those ones for which the associated Hilbert schemes behave as in the general case of the corresponding deformation family. We refer to these Fanos over $\mathbb K$ as {\it{$\mathbb K$-standard Fanos}}. 
\end{rem}

\subsection{Basic arithmetic on the Hilbert Schemes of Lines where $\rho(X)=\iota(X)=1$ or $\rho(X)=1$, $\iota(X)=2$}
If $X$ is a $\mathbb K$-standard Fano threefold, a $\mathbb K$-rational line is simply a $\mathbb K$-rational point of $\Sigma(X)$. The following Proposition is our first result about the arithmetic of Fanos.

\begin{prop}\label{aritmetica1} Let $X$ be a smooth $\mathbb K$-standard Fano threefold.
Assume that $\rho(X)=1=\iota(X)=1$ and $g\geq 4$ or $g=3$ and $X$ is a quartic threefold or $\rho(X)=1$, $\iota(X)=2$ and $2\leq d(X)\leq 3$. Then the set of $\mathbb K$-rational lines in the Fano scheme of lines $\Sigma(X)$ is degenerate. If $\rho(X)=1$, $\iota(X)=2$ and $d=4$ then it is dense.
If $\rho(X)=1$, $\iota(X)=2$ and $d=5$ then it is dense.
\end{prop}
\begin{proof} If $\rho(X)=1=\iota(X)=1$ and $g\geq 4$ or $g=3$ and $X$ is a quartic threefold then $\Sigma(X)$ is a smooth curve of general type. Hence the claim follows by Faltings' theorem. If $\rho(X)=1$, $\iota(X)=2$ and $d=2$ then $X$ is a quartic double solid and the finiteness of the rational points on the corresponding surface $\Sigma_X$ is the main result of \cite{CoZucc-secondo}. 
 If $\rho(X)=1$, $\iota(X)=2$ and $d=4$ then $\Sigma(X)$ is an abelian variety, hence its rational points  are potentially dense (simply take any algebraic point generating a Zariski-dense subgroup). 
If $\rho(X)=1$, $\iota(X)=2$ and $d=5$ then $\Sigma(X)=\mathbb P^2$.
\end{proof}
\subsection{Hilbert Schemes of Conics}
We turn our attention to the Hilbert scheme of conics $S(X)$. We follow the standard  partition among the families of Fanos.

\subsection{Hilbert Schemes of Conics where $\rho(X)=1$, $\iota(X)=2$ and $H$ is very ample}
For the case $\rho(X)=1$, $\iota(X)=2$ and $d=2$, that is the case of quartic double solid $\pi\colon X\to\mathbb P^3$, where $H$ is not very ample, we know that $S(X)$ has a morphism over the dual of $\mathbb P^3$ whose general fiber is a disjoint union of $126$ rational curves: \cite[Lemma 2.5. and the full Section 2]{Ha}.  The arithmetic question of the denisty of rational points in this four-fold is left open.

For the remaining case it holds:

\begin{prop} Let $X$ be a smooth Fano threefold with $\rho(X)=1$, $\iota(X)=2$ and $H$ is very ample. Then for the Fano scheme $S(X)$ of conics it holds:
\begin{enumerate}
\item if $d=3$ then $S(X)$ is  $\mathbb P^2$-bundle over $\Sigma(X)$;
\item if $d=4$ then $S(X)$ is a $\mathbb P^3$-bundle over a curve of genus $2$;
\item if $d=5$ then $S(X)$ is isomorphic to $\mP^4$.
\end{enumerate}
\end{prop}
\begin{proof} For the case $d=3$ see \cite[Proposition 3.3]{HRS}. See c.f.\cite[Proposition 2.3.8]{KPS}.
\end{proof}

\subsection{Hilbert Schemes of Conics where $\rho(X)=\iota(X)=1$}
As we were reviewing the literature on $S(X)$ it seemed to us that the behaviour of the Abel-Jacobi map from $S(X)$ to the intermediate Jacobian $J(X)$ of $X$ is still not fully described. We do not need to clarify this point for our aims in this work. Nevertheless we consider it appropriate to point out the reader that for $\rho(X)=\iota(X)=1$ it makes sense to consider four subcases: $3\leq g\leq 5$, $g=6$, $g=9$, $g\geq 7$ and $g\neq 9$, as far as the Abel-Jacobi map it concerns.


 \subsubsection{The case g=3} 
 The description and the geometry of $S(X)$ is given in \cite{CMW}.
 \begin{prop}\label{genere 3}
 If $X\subset \mP^4$ is a smooth quartic threefold then $S(X)$ is a surface of general type
 with the following invariants:
 $$
 K_{S(X)}^2= 341040, c_2(S(X)=172704, p_g(S_X)=42841, q(S_X)=30.
 $$
\end{prop}
%
 
 \subsubsection{The case g=4}  
  A Fano threefold of genus four is the complete intersection $X = Q \cap W$ of a quadric $Q$ and a cubic hypersurface $W$ in $\mP^5$. Let $T$ be the rank three tautological vector bundle on $\mathbb G(3,6)$. The Hilbert scheme $\sH_2$ of conics in $\mP^5$ is the total space of the projective bundle $\mP({\rm{Sym}}^2T^\vee)\to \mathbb G(3,6)$. On $\sH_2$ we have the tautological sequence
  $$
  0\to   \sO_{\sH_2}\to \pi^\star{\rm{Sym}}^2T^\vee\to\sQ\to 0
  $$
  and the equation $f_Q$, of $Q$ defines a section  $\sigma_Q$ of $\sQ$ whose zero-locus is the set of conics contained in $Q$.
  By the exact sequence 
  $$
  0\to\sO_{\sH_2}(-1)\otimes \pi^\star T^\vee \to\pi^\star{\rm{Sym}}^3T^\vee\to\sM\to 0
  $$
  it remains defined the vector bundle $\sM$ and the equation $f_W$ gives a section $\sigma_W$ of $\sM$ whose zero locus is the set of conics contained in $W$. Since $\sQ\oplus\sM$ is globally generated and $\sigma_Q\oplus\sigma_W$ is a sufficiently general section, by Bertini's theorem $\Sigma(X)$ is a smooth surface. More precisely:  
 \begin{prop}\label{genere 4}
 If $X$ is a general complete intersection $X = Q \cap W$ of a smooth quadric $Q$ and a smooth cubic hypersurface $W$ in $\mP^5$ then $S(X)$ is a surface of general type with the following invariants:
 $$
 K_{S(X)}^2= 23355, c_2(S(X))=11961, p_g(S(X))=2942, q(S(X))=20.
 $$
 \end{prop}
 \begin{proof} See \cite[Corollary 10]{IM2}.
 \end{proof}
 \subsubsection{The case $g=5$}
 If $X$ is a general complete intersection of three quadrics a direct proof about the smoothness of $S(X)$ is in \cite{PB}.
 \begin{prop}\label{genere 5}
 If $X$ is a complete intersection $X = Q_1 \cap Q_2\cap Q_3$ of three quadrics in $\mP^5$ then $S(X)$ is a surface of general type
 with the following invariants:
 $$
 K_{S(X)}^2= 3376, c_2(S(X))=1760, p_g(S(X))=441, q(S(X))=14.
 $$
 \end{prop}
 \begin{proof} See c.f. \cite[p.p. 7, 8]{IM1}.
 \end{proof}
 \subsubsection{Abel-Jacobi map in the case $3\leq g\leq 5$}
\begin{thm}\label{quarticacubicaquadricatrequadriche}
 Let $X$ be a $\mathbb K$ standard Fano threefold whit $\rho(X)=\iota(X)=1$ and $3 \leq  g \leq 5$. Then the Abel-Jacobi map from $S(X)$ to the intermediate Jacobian of $X$ is generically finite onto the image.
 \end{thm}
 \begin{proof} See \cite[Theorem 7.1 and Lemma 7.2]{LT} and Remark \ref{moltoimportante}.
 \end{proof}

\subsubsection{The case $g=6$}
Fano threefolds $X$ of genus $6$ are intersection of the Grassmannian variety $\mathbb G(2,5)\subset\mP^9$
 with a linear subspace of dimension $7$  and a quadric. In other words $X= G(2,5)\cap H_1\cap H_2\cap Q$ where $[H_1], [H_2]\in \mP^{9^\vee}$ and $Q$ is a quadric. We assume that $G_4:= \mathbb G(2,5)\cap H_1\cap H_2$ is smooth. 

\begin{prop}\label{genere 6} Let $X$ be a smooth Fano threefold with $\rho(X)=\iota(X)=1$ and genus $ g(X)=6$. Up to a codimension $1$ loci in the moduli space of $X$, the Fano scheme of conics $S(X)$ is a smooth irreducible surface of general type which is the blow-up of a surface $S$
with $q(S) = 10$, $p_g(S) = 101$, $h^1(S,\Omega^{1}_{S} ) = 220$, $c_2(S) = 384$.
\end{prop}
\begin{proof} See \cite[Prop 0.1, Prop. 0.5]{Lo}.
\end{proof}
\begin{rem}\label{abeljacobisei} We do not know if the Abel-Jacobi map from $S(X)$ to the intermediate Jacobian of $X$ can be factorised through a curve for some loci in the deformation class of these Fanos.
\end{rem}

\subsubsection{The case $g\geq 7$}
As far as the Abel-Jacobi map it concerns definitely for the case $g=9$ $S(X)$ cannot be mapped generically finitely on the intermediate Jacobian of $X$. It holds:
\begin{prop}\label{genere maggiore 6} Let $X$ be a smooth Fano threefold with $\rho(X)=\iota(X)=1$ and genus $g = g(X) \geq 7$. Then the Fano scheme of conics $S(X)$ is a smooth irreducible surface. More precisely,
\begin{enumerate}
\item if $g=7$ then $S(X)$ is the symmetric square of a smooth curve of genus $7$;
\item if $g=8$ then $S(X)$ is a minimal surface of general type with $q(X)=5$, $p_g(X)=10$, and $K_{X}^2 =45$;
\item if $g=9$ then $S(X)$ is a ruled surface isomorphic to the projectivization of a simple rank $2$ vector bundle on a smooth curve of genus $3$;
\item if $g=10$ then $S(X)$ is an abelian surface;
\item if $g=12$ then $S(X)$ is isomorphic to $\mP^2$.
\end{enumerate}
\end{prop}

\begin{rem} To know the relation between $S(X)$ and the intermediate Jacobian of $X$ in the case $2\leq g\leq 12$, $g\neq 11$ see \cite[Section 3.1]{IM1} and the literature quoted therein.
\end{rem}

We think that the following is an open problem:

\begin{prob}\label{problemaaperto} For each deformation families find the loci (it could be empty) where $S(X)$ maps to a curve of the intermediate Jacobian of $X$ via the Abel-Jacobi map.
\end{prob}

\subsection{Basic arithmetic on the Hilbert Schemes of conics} 
If $X$ is a $\mathbb K$-standard Fano threefold, a $\mathbb K$-rational conic is simply a $\mathbb K$-rational point of $S(X)$.  

\begin{prop}\label{aritmetica2} Let $X$ be a $\mathbb K$-standard Fano threefold.
If $\rho(X)=1=\iota(X)=1$ and $4\leq g\leq 9$ ($g\neq 11$) or $g=3$ and $X$ is a quartic threefold or if $\rho(X)=1$, $\iota(X)=2$ and $3\leq d(X)\leq 4$ then the set of $\mathbb K$-rational conics in the Fano scheme of conics in $S(X)$ is degenerate. If $\rho(X)=1$, $\iota(X)=2$ and $d=5$ or $\rho(X)=1=\iota(X)=1$ and $g=10$ or $g=12$ the set of $\mathbb K$-rational conics is dense in $S(X)$
 \end{prop}
\begin{proof} It is a standard case to case analysis which used the same results used to show Proposition \ref{aritmetica1}.
\end{proof}
We think that notwithstanding its proof is rather easy, Proposition \ref{aritmetica2} and Proposition \ref{aritmetica1} constitute an important first step in understanding the arithmetic of Fano threefolds. Moreover they show clearly that to study potentially density of integral points by a brute use of Hilbert schemes seems to be of little help in almost all the cases.

\section{Density of integral points on certain varieties}

\subsection{Integral points and the puncturing conjecture} 

Closely related to the construction of integral points on affine varieties is the following conjecture proposed by B. Hassett and Yu. Tschinkel \cite{Hassett-Tschinkel}, which we name \lq the puncturing conjecture\rq:

{\bf Conjecture} [{\tt The puncturing conjecture}]. {\it Let $X$ be a smooth projective variety over a number field $\K$ such that the set $X(\K)$ of rational points is Zariski-dense. Let $Y\subset X$ be a closed sub-variety of codimension $\geq 2$. Then the integral points of $X\setminus Y$ are potentially dense}.

\smallskip
The first instance of this conjecture  is represented by the case of a surface $X$ and a finite set $Y$, hence the name we propose for the conjecture.

The above conjecture does not concern directly affine varieties, since the complement of a higher codimensional closed set in a projective variety is never affine: however, its application to varieties obtained as Hilbert schemes of lines (or of conics) on Fano threefolds will provide density results on integral points on affine open sets of such threefolds.

The result is elementary  for $X=\mP^2$, and $Y$ an arbitrary finite set (we give a proof below  for lack of a reference).
As mentioned by Hassett and Tschinkel, this is known to hold also for products of elliptic curves, but is still unknown for simple abelian varieties (even in the case where $Y$  is a single point). Here is a result strenghtening  the puncturing conjecture for $X=\mP^2$:

\begin{lem}\label{lemma-punture}
Let $Y\subset \mP^2$ be the algebraic closed set formed by the union of three lines in general position and a finite set of points. Then the integral points on $X^0:=\mP^2\setminus Y$ are potentially dense.
\end{lem}

Clearly this lemma implies the puncturing conjecture for the plane.

\begin{proof}
We can choose   coordinates in $\mP^2$ so that one of the three lines lies at infinity and the other two are expressed, in affine coordinates $x,y$, by the  equations $x=0$ and $y=0$.  Then the integral points on $\mP^2$ with respect to the three lines are represented, in these affine coordinates, by pairs of units $(\alpha,\beta)\in \sO_S^*\times \sO_S^*$. Let $P_1=(a_1,b_1),\ldots,P_n=(a_n,b_n)$ be a finite set of points in $\mG_m^2$. Afetr a suitable enlargement  of $S$  shall construct a Zariski-dense set of points  $(\alpha,\beta)\in \sO_S^*\times \sO_S^*$ which are integral with respect to $P_1,\ldots,P_n$. This means that $\alpha-a_i$ and $\beta-b_i$ are a units for every $i=1,\ldots,n$. Let us enlarge $S$ and $\K$ so that there exists an element $\alpha \in   \sO_S^*$ of infinite order in the multiplicative group and integral with respect to $a_1,\ldots,a_n$. Still enlarging $S$ if necessary, we can   construct a point $\beta\in\sO_S^*$ with the same integrality property, this time with respect to $b_1,\ldots,b_n$. Take any power $\alpha^m$ of $\alpha$ distinct from any of the $a_1,\ldots,a_n$ (in the case some of the $a_i$ belongs to the group generated by $\alpha$); this unit $\alpha^m$ can be congruent to one of the $a_i$ only modulo a finite number of places. Let ${\mathcal I}$ the product of the prime ideals modulo which some of the $a_i$ is congruent to $\alpha^m$. Since $\beta$ is a unit, it does not belong to $\mathcal I$, hence has finite order, say $k\geq 1$, modulo $\mathcal I$; then the points $\beta^{kn+1}$, for $n=1,2,\ldots$, are all congruent to $\beta$ modulo $\mathcal{I}$, so in particular they are not congruent to any of the $b_i$, $i=1,\ldots,n$. Then clearly the points $(\alpha^m,\beta^{nk+1})$ do not reduce to any of the points $(a_i,b_i)$ modulo any valuation of $S$.  
\end{proof}

Note that this proof extends {\it verbatim} to any pair $(X,Y)$ where $X$ is a product of semi-abelian varieties, for instance to the product of two elliptic curves, or of an elliptic curve by a torus, and $Y$ a finite set.

\subsection{The fully integral curves}\label{S.fully-integral}
 We follow ideas of F. Beukers, further developped by B. Hassett and Yu.Tschinkel \cite{Hassett-Tschinkel},  aimed at constructing dense sets of integral points on higher dimensional varieties. The starting point will be Lemma \ref{lemma-beukers} below. 
\subsubsection{The notion of $S$-integral point with respect to the divisor $D$}
We fix a projective algebraic variety $X\subset \mP^N$ embedded once for all in a projective space, defined over a number field $\K$. We assume $D\subset X$ to be a closed proper sub-variety, also defined over $\K$. We fix a finite set of places $S$ of $\K$ and denote by $\sO_S$ the corresponding ring of $S$-integers.

Recall that a rational point $x\in X^o:=X\setminus D$ is said to be $S$-integral with respect to $D$ if for no prime ideal $\mathcal P$ of $\sO_S$ the point $x$ reduces modulo $\mathcal P$ to a point of $D$. Whenever $D={y}$ is a point, we say that $x$ and $y$ are {\it{coprime}} if $x$ is integral with respect to ${y}$. Note that this is a symmetrical condition.

We take an algebraic variety $X$ of dimension $n\geq 2$ (usually the dimension will be $3$ in this work) and we are interested in constructing integral points on $X^o$ by finding in  $X^o$ a $n-1$-dimensional family of curves each admitting infinitely many integral points.
\subsubsection{Beukers Lemma}
As we said, whenever $L$ is a rational curve with one or two points at infinity, its integral points  are {\it potentially} dense. However, in our construction we need that for a {\it given} ring of $S$-integers, infinitely many of the curves we consider have infinitely many $S$-integral points. The main tool to search $S$-integral points is represented by the following:

\begin{lem}\label{lemma-beukers}
Let $L\subset X$ be a smooth rational curve which does not reduce to a curve in $D$ modulo any place outside $S$. 

\begin{enumerate}[(a)]
\item If $L\cap D$ consists in a single point then the set  $L(\sO_S)$ of $S$-integral points of $L$ is infinite. 
\item If $L\cap D$ consists in a pair of coprime rational points $A,B$  and the group of $S$-units $\sO^*$ is infinite, then $L(\sO_S)$ is infinite.
\item If $L\cap D$ consists in a pair of non-coprime rational points $A,B$  and the group of $S$-units $\sO^*$ is infinite, then $L(\sO_S)$ is infinite if and only if it is non-empty.
\item If $L\cap D$ consists in a pair of (conjugate) quadratic points over $\K$, then $L(\sO_S)$ is infinite if and only if it is non-empty.
\end{enumerate}
\end{lem}

For a proof, see \cite{Beukers}.
\medskip

\subsubsection{Case study for Beukers' Lemma}
Let us explain the significance of the above lemma via a concrete examples for each of the four different cases.
\smallskip

({\it a}) Suppose $X$ is the projective plane and $D$ a line (viewed as the line at infinity, so that $X\setminus D=\A^2$).  
To simplify, suppose that $\sO_S=\mathbb Z$: Let $ax+by=c$ be an equation of the affine line $L\setminus(L\cap D)$, where $a,b,c$ are integers with no common factor. The condition that $L$ does not reduce to the line at infinity $D$ modulo any prime amounts to saying that no prime divides both $a$ and $b$. In that case, the Diophantine equation $ax+by=c$ is knwon to admit infinitely many integral solutions.\smallskip

({\it b,c})  Suppose again that $X$ is the projective plane and $D$ a line, but now $L$ is a conic, not tangent to $D$, so that $L\cap D$ consists in two points. Suppose these points are rational. The affine conic will be a hyperbola of equation
\begin{equation*}
(ax+by) (cx+dy)=n
\end{equation*}
for some  $S$-integers $a,b,c,d$ and non-zero integer $n$. If $ad-bc$ is a unit in $\sO_S^*$,  every pair of integers $(\xi,\eta)$ can be written in the form $\xi=ax+by$ and $\eta=cx+dy$; then, to every factorization of $n=\xi\cdot \eta$ (and there exist infinitely many such factorizations if the group $\sO_S^*$ is infinite) there corresponds an $S$-integral points on the hyperbola. Note that the condition that $ad-bc\in \sO_S^*$ is equivalent to requiring that the two points at infinity do not coincide modulo any prime of $\sO_S$. However, the simple example of the equation $(x+y)(x-y)=2$ over $\Z$, 
or even over $\Z[1/3]$ whose group of units is infinite, shows that there might exist no solutions at all. On the contrary, the example of the equation $x^2-y^2=3$ over the ring $\Z[1/3]$, or $\Z[1/n]$ for any integer $n>1$, shows that there might exist infinitely many solutions. \smallskip

({\it d}) The case of quadratic integral points at infinity leads to   Pell-like equations, of the form
\begin{equation}\label{eq-Pell-like}
x^2-dy^2=n,
\end{equation}
where $d\in \sO_S$ is a non-square. When $n=1$ (Pell's equation)  the solutions correspond to the elements of norm $1$ in the ring $\sO_S[\sqrt{d}]$ and there are always infinitely many of them.
In some cases (e.g. $d=3, n=2$ and $\sO_S=\Z$) Equation \ref{eq-Pell-like} admits no solutions. If, however, there exists one solution, then there exist infinitely many others, since the set of solutions is acted on by the infinite group of linear transformations defined over the integers conserving the quadratic form on the left-hand side above; the   infinitude of this group corresponds to the infinitude of the set of solutions for $n=1$ .

\smallskip
\subsubsection{The notion of fully integral rational curve}
Hence, under the provisio that the group $\sO_S^*$ is infinite, the cases where we can conclude that the curve $L$ contains infinitely many integral points without checking the presence of a single integral point are cases (a) and (b). We agree to call {\it fully integral} a rational curve which never reduces to a curve at infinity and satisfies  the property (a) or (b) of Beukers' Theorem, see Lemma \ref{lemma-beukers}.

A smooth rational curve on a variety $X^o=X\setminus D$ which merely satisfies the condition of having at most two points at infinity and not reducing to a curve in $D$ modulo any prime is called {\it integral}. We stress that integral lines could give infinite $S$-integral points with respect to $D$ if for some extra reasons one could show that condition $c)$ or condition $d)$ holds for (some of) them. We like to think that this quest on the geometry of $(X,D)$ in order to gain condition $c)$ or condition $d)$ of Beukers Lemma \ref{lemma-beukers} will be able to revitalize the study of delicate questions of classical projective geometry. We point out the reader that whenever $\sO_S^*$ is infinite, the set  $L(\sO_S)$ is infinite whenever it is non-empty, for any rational integral curve $L$.

\subsection{The complement in $\mP^3$ of the union of two quadrics}
\label{The complement}

In this subsection we shall consider the integral points on the complement of two quadrics in $\mP^3$, whose union has normal crossing singularities. 

\subsubsection{The case of a smooth quadric and two planes}
We first treat the easy case where one of the two quadrics is reducible, i.e. the union of two planes. First we show:
\begin{lem}\label{quasiovvio} Let $Q\subset\mathbb P^3$ be a smooth quadric and let $h_1,h_2$ be two smooth  hyperplane sections intersecting properly (i.e. in two points). Let $l_q,r_q$ be the two lines lying on $Q$ passing through of the two intersection points $q\in h_1\cap h_2$. Then the integral points on $Q^o:=Q \setminus(l_q\cup r_q\cup h_1\cup h_2)$ are potentially dense.
\end{lem}
\begin{proof} Let $q\in h_1\cap h_2$ and let $l_q, r_q$ be the two lines in $Q$ through $q$. Let $f\colon S\to Q$ be the blow up at $q$, $E$ the $-1$-curve above $q$ and $h_1'$, $h_2'$, $l'_q, r'_q$ the strict transform of $h_1$ and, respectively $h_2$, $l_q, r_q$. Let $f'\colon S\to \mP^2$ be the contraction of the two $-1$ curves $l'_q, r'_q$. 
The composition $f' \circ f^{-1}$  is a well defined isomorphism $g \colon Q^o \to \mP^2 \setminus(L\cup L_1\cup L_2)$ where $L:=f'(E)$, $L_i:= f'(h'_i)$, $i=1,2$ are lines in general position. By Lemma \ref{lemma-punture} the claim follows.
\end{proof}

\begin{thm}\label{quadricaeduepiani}
Let $Q$ be a smooth quadric surface. Let $H_1,H_2$ be two planes whose intersections with $Q$ are smooth conics intersecting properly (i.e. in two points).
Then the integral points on $X^o:=\mP^3\setminus(Q\cup H_1\cup H_2)$ are potentially dense.
\end{thm}

\begin{proof}
Set, for $i=1,2$, $h_i:=H_i \cap Q$. Let $\{q,\widehat q\}:= h_1\cap h_2$.  Let $l_q,r_q$ be the lines on $Q$ passing through $q$.  Enlarge the ring of $S$-integers so that the integral points on $Q$ with respect to $h_1\cup h_2\cup l_q\cup r_q$ are Zariski-dense and the group of units is infinite.

Let $p$ be an integral points on $Q\setminus( h_1\cup h_2\cup l_q\cup r_q)$ and let $s=s(p)$ be the line joining $q$ to $p$. We claim that $s$ does not reduce to any line on $Q$ not to a line on $H_1\cup H_2$. Indeed, since $q\in s$, the only lines on $Q$ wich might be congruent to $s$ modulo some prime are $l_q,r_q$, but this is excluded by the fact that $p$ is integral with respect to these lines. Since $p$ is integral also with respect to $H_1,H_2$, the line $s$ cannot reduce to $H_1\cup H_2$.

Now, $s$ is fully integral with respect to $D:=Q\cup H_1\cup H_2$ since it does not reduce to a line on $D$ and $D\cap s$ consists on the coprime points $q,p$. By Lemma \ref{lemma-beukers}, $s$ contains infinitely many $S$-integral points.

Since $p$ varies in a two-dinesional set, and any two distinct such lines $s(p_1),s(p_2)$ always intersect only at $q$, the union of such lines is Zariski-dense on $\mP^3$ and the result follows.
\end{proof}

\subsubsection{Interlude: the case of a smooth cubic and a plane}\label{cubic+plane}
Let us consider still  another quartic case, the union of a smooth cubic surface and a plane; however, to provide the sought density result for integral points we shall need that their intersection contains a line:

\begin{thm}\label{cubicandplanewithacommonline}
Let $D\subset \mP^3$ be the union of a smooth cubic surface $V$ and a plane $H$. Suppose that the intersection $V\cap H$ contains a line $L$. Then the integral points on $\mP^3\setminus D$ are potentially dense.
\end{thm}

 \begin{proof} Let $H\cap V=L\cup C_H$ where $L\cap C_H=\{A,B\}$.
 The plane $H$ is tangent to $V$ at the two points $A,B\in L$. We can enlarge the set $S$ so that $L$ contains infinitely many integral points with respect to $A$ and $B$ and that $\sO_S^*$ is infinite.
Let $L'$ be any tri-tangent line to $V$ of $V\cap H$, not contained in $D$; enlarge again $S$ so that $L'$ does not reduce to a line of $D$ modulo any prime. We set $\{p_{L'}\}:=L'\cap V$.

Let $x\in L$ be a point of $L$ which is integral with respect to $a,b$. The tangent plane $T_xV$ to $V$ at $X$ contains $L$.
Then the intersection $T_xV\cap V$ is the union of a conic $C_x$ and the line $L$. It holds that $C\cap L:=\{x, x'\}$ where $x'$ is another point of $C$. Since $x$ is integral with respect to $a,b$ then $T_xV$ does not reduce to $H=T_aV=T_bV$ modulo any prime. 
Now we set $L'\cap T_xV:=\{y\}$. Note that $L'\cap L=\emptyset$. We claim that the pointy $y$ does not reduce to $D$ for any prime outside $S$. Indeed if $y\in L'$ reduces to $D=V\cup H$ then $y$ reduces to $p_{L'}$ since $L'\cap (H\cup V)=\{p_{L'}\}$. Then $L'$ becomes tritangent to $y$ and $y$ reduces to $L\cup C_x$. Hence $T_xV$ becomes the unique plane which contains $L$ and $p_{L'}$; but this plane is $H$. Then $T_xV$ reduces to $H$. Thus $T_xV$ reduces to $T_aV$ or to $T_bV$. In particular $x$ reduces to $a$ or $b$: a contradiction. Since the point $y$ of the plane $T_xV$ does not reduce to $D$ it does not reduce to $D_{|{T_xV }}=L\cup C_x$. Now we consider the line $L_{x,y}$ through the points $x$ and $y$. This line intersects $D$ in the point $x$ and in another point $z\in C_x$, but it never reduces to $D$. Moreover it contains the point $y$ which is integral with respect to $x$ and $z$. There are two sub-cases. If $x$ and $z$ are coprime, by Lemma \ref{lemma-beukers} (b) then $L_{x,y}$ contains infinite many integral points with respect to $D$. If $x$ and $z$ are not coprime, nevertheless $L_{x,y}$ contains the point $y$ which is integral with respect to \lq\lq the two points at infinite\rq\rq $x$ and $z$, and so by Lemma \ref{lemma-beukers} (c), even in this case $L_{x,y}$ contains infinite many integral points with respect to $D$.

Finally moving the point $x$ among the integral points of $L\setminus\{A,B\}$, that is, moving the plane $W=T_xV$ in the pencil of planes containing the line $L$, we construct a one-dimensional family of surfaces each with a Zariski-dense set of integral points with respect to $D$. This implies our claim.
 \end{proof}

We propose the following:

\begin{conj}\label{complement by plane and cubic} Let $D=\Pi\cup \Theta$ where $\Pi$ and $\Theta$ are respectively a plane and a cubic surface in $\mathbb P^3$. Suppose that $D$ has normal crossing singularities. Then the   integral points on $\mP^3$ with respect to $D$ is potentially dense.
\end{conj}

\subsection{The case of the union of two quadrics with reducible intersection} 
Let $Q_1,Q_2 \subset\mP^3$  be two smooth quadric surfaces, defined over a number field, such that their intersection is the union of two smooth conics. 
The complement $X^0:=\mP^3\setminus(Q_1\cup Q_2)$ is log-CY; we can prove in this case  the potential density of integral points on $X^o$:

\begin{thm}\label{two-quadrics-two concis}
Let $X^o=\mP^3\setminus(Q_1\cup Q_2)$ as above. The integral points on $X^o$ are potentially dense.
\end{thm}

\begin{proof}
We construct a Zariski-dense set of fully integral {\it conics}. Let $C_1,C_2$ be the two conics whose union gives $Q_1\cap Q_2$, $H_1,H_2$ be the corresponding planes (i.e. $H_i\cap Q_1=H_i\cap Q_2=C_i$) and set $\{q, \hat{q}\}=C_1\cap C_2$. Consider the pencil of planes generated by $H_1,H_2$ (i.e the pencil of planes containing $\{q,\hat{q}\}$). Each plane $H$ in the pencil intersects $Q_1\cup Q_2$ in a pair of bitangent conics $C^H_1, C^H_2$ (which coincide precisely in the two cases when the  plane $H$ equals $H_1$ or $H_2$). Now, in the mentioned pencil of planes there are infinitely many integral points with respect to $H_1,H_2$; choose one such point $H$, i.e. a plane which never reduces to $H_1$ nor to $H_2$. Then the two conics $C^H_1,C^H_2$ do not reduce to each other modulo any prime. Then the pencil of conics on $H$ generated by $C^H_1,C^H_2$ contains infinitely many points which are integral to $C^H_1$ and $C^H_2$; all such conics, even the reducible ones, are fully integral curves with respect to $Q_1\cup Q_2$. Their union being Zariski-dense in $H$, and by varying $[H]$ in the integral points of the pencil enable us to construct a Zariski-dense set of integral points, still exploiting Lemma \ref{lemma-beukers}.
\end{proof}

\subsection{The case of the union of two quadrics which intersect properly}
Let $Q_1,Q_2 \subset\mP^3$  be two smooth quadric surfaces, defined over a number field such that their intersection is a smooth curve $E=Q_1\cap Q_2$ of genus one. 

We say that a line is bitangent to $D=Q_1\cap Q_2$ if it touches $D$ with multiplicity $\geq 2$ at each of its intersection points. These bitangent lines can be of two types:\smallskip

\begin{enumerate}
\item the bisecants to $E=Q_1\cap Q_2$ intersect $D$ only at singular points, so they are bitangents. They form a surface, which can be described as the symmetric square of $E$. It is a ruled surface with base $E$.
\item The lines which are tangent to both $Q_1$ and $Q_2$. These are the {\it{proper bitangents}}. They form a surface which we denote $\Sigma$.
\end{enumerate}

\smallskip

The two families, viewed as sub-varieties of the Grasmannian $G(2,4)$, intersect in a curve, denoted by $\tilde{E}$ which is formed by the quadritangents.

\smallskip

Let $\Sigma$ be the set of lines in $\mP^3$ which are  tangent to both $Q_1,Q_2$. This variety is a surface, which is naturally embedded into the Grasmannian $G(2,4)$ parametrizing lines in $\mP^3$. As mentioned, the surface $\Sigma$ contains a curve, denoted by $\tilde{E}$ and isomorphic to $E$, consisting of the lines which are tangents to both $Q_1$ and $Q_2$ at one of their intersection points. Note that  $\tilde{E}$ contains sixteen  \lq special\rq points, corresponding to the eight lines c(four for each ruling) contained in $Q_1$ and tangent to $Q_2$ and the four lines contained in $Q_2$ and tangent to $Q_1$. 

In the dual projective space, parametrizing the planes in $\mP^3$, the points corresponding to the bitangent planes form the intersection of two smooth quadric surfaces, namely the duals to $Q_1$ and $Q_2$. This is again a genus one curve. Denote by $E'$ its image in $Q_1$ (to each bitangent plane associate its point of tangency with $Q_1$).

We shall prove the follwing:

\begin{lem}\label{bitangenti-kummer}
The surface $\Sigma$ is a (Kummer) K3-surface.
\end{lem}

\begin{proof}
Let us consider variety $S'$ of  pairs $(p,[l])\in Q_1\times \Sigma$ such that $p\in [l]$. The canonical projection $S'\to \Sigma$ has degree one and has eight one-dimensional fibers, namely the points $[l]$ such that $l$ is included in $Q_1$ (and tangent to $Q_2$). Now, the natural projection $S'\to Q_1$ has degree $2$. It ramifies over $E\cup E'$. These   curves are both smooth of bidegree $(2,2)$ and intersects at eight distinct points; then $S'\to Q_1\simeq \mP^1\times \mP^1$ is the (unique) double cover ramified over $E\cup E'$; it has singularities over the intersection $E\cap E'$; a simple calculation of its canonial divisor shows that it is birational to a  K3-surface with eight singularities. It can be desingularized by blowing-up the eight points of $E\cap E'$ and replacing $S'$ by the corresponding degree-2 cover of this blown-up surface.  

We finally obtain a diagram
\begin{equation}\label{Diagram1}
\begin{matrix} 
S & \longrightarrow & S' & \longrightarrow & \Sigma \cr
\downarrow  &{} & \downarrow  & {} & {} \cr
\widehat{\mP^1 \times \mP^1}& \longrightarrow & Q_1 &{} &{}\cr
 \end{matrix}
\end{equation}
where the arrow $\widehat{\mP^1 \times \mP^1} \rightarrow   Q_1$ denotes the blow-up of $Q_1\simeq \mP^1\times \mP^1$ at the eight points of $E\cap E'$. 

\end{proof} 

The surface $S$ is a smooth K3-surface, and is birational to the surface of bitangents $\Sigma$. Concretely, it can be defined as a set of triples
\begin{equation}\label{S-concrete}
S=\{(P_1,P_2,[l])\in Q_1\times \Q_2 \times \Sigma\, :\, P_1\in l\cap Q_1, \, P_2\in l\cap Q_2\}.
\end{equation}
Generically, $l$ determines both $P_1$ and $P_2$, with sixsteen exceptions, corresponding to the singularities of $\Sigma$ and giving rise to sixsteen $(-2)$-curves on $S$.

We now construct an elliptic fibration on $S$. The two curves $E,E'\subset Q_1$ have the same bidegree, precisely  $(2,2)$, hence they are linearly equivalent. Namely, they generate a pencil, whose base locus is made of the mentioned eight points. Given a rational function $f\in \K(Q_1)$ whose divisor is $E-E'$, a model for $S$ is given by the equation $y^2=f(x)$, where $x\in Q_1$. Every curve in the pencil has a reducible pre-image in $S$, whose components have arithmetic genus $1$; generically, these components are smooth. Clearly, $E$ and $E'$ are fibers of $f$, viewing $f$ as a morphism $S\to \mP^1$. This morphism admits a Stein factorization 
\begin{equation*}
S\to \mP^1\to \mP^1
\end{equation*}
where the last morphism has degree $2$ and ramifies over $f(E), f(E')$; the first one $h\colon S\to\mP^1$ has connected fibers. It fits in a diagram
\begin{equation}\label{Diagram2}
	\xymatrix{
		S\ar[r]\ar^{h}[d] & \widehat{\mP^1 \times \mP^1} \ar^{}[d]\\
		\mP^1\ar^-{z\to z^2}[r]& \mP^1=\langle [E],[E']\rangle\\
	}
\end{equation}
 where the second vertical arrow is the morphism induced on  $\widehat{\mP^1 \times \mP^1}$ by the pencil generated by $E,E'$ over $Q_1\simeq \mP^1\times \mP^1$. The fibration $h\colon S\to \mP^1$ has general fiber of genus $1$ and it admits eight sections, corresponding to the eight points of intersections $E\cap E'$. Taking one of them for the origin, we obtain a fibration in elliptic curves (i.e. genus one curves with a marked point). 
\begin{thm}
The rational points on $\Sigma$ (hence on $S,S'$) are potentially dense.
\end{thm}
\begin{proof}
The fact that ellitpic $K3$ surfaces always admit a potential dense set of integral points is a theorem of Bogomolov and Tschinkel. Moreover, in this special case it can be proved that the above elliptic surface has positive Mordell-Weil rank, leading to a simple way of producing a Zariski-dense set of rational points.

We provide another argument: the surface $S$ projects to both $Q_1$ and $Q_2$ with morphisms of degree $2$.  These morphisms are associated to involutions $\tau_1,\tau_2$, which can be described as follows: take a point   $s=(P_1,P_2,l)\in S$ (using the description \eqref{S-concrete} for the surface $S$) and associate the point $\tau_1(s)=(P_1,P_2',[l'])$ where $[l']$ is the other bi-tangent that can be drawed from $P_1$ and $P_2'$ is the intersection point of $l'$ with $Q_2$. Analogously one defines $\tau_2$. The group generated by $\tau_1,\tau_2$ is infinite, turning $S$ into a ``Markov-like K3-surface', as first studied by J. Silverman \cite{silverman}. (The name derives from the affine cubic surface of equation $x^2+y^2+z^2=3xyz$, a Diophantine equation first considered by Markov; this surface is endowed with a discrete group of automorphisms, generated by the three involutions corresponding to the degree-two projections  on two coordinates). 

The orbit of a generic point under this group is dense, so one can construct in this way a dense set of rational points.

\smallskip

Still another argument derives from the fact that $S$ is a Kummer surface, quotient of an abelian surface, where rational points are clearly Zariski-dense. 
\end{proof}

The fact that $S$ is covered by an abelian surface derives from the theory of ellipsoidal billiard.  Given a shot in an ellipsoidal billiard $Q$, i.e. a segment joining two (real) points on the surface $Q$, there exist two other confocal quadrics $Q_1,Q_2$ -named caustics - such that each other segment of the billiard trajectory defined by the first shot is tangent to both $Q_1$ and $Q_2$. In chapter 7 of the book by V. Dragovic and M. Radnovic \cite{DR} it is explained that ordered segments of trajectories corresponding to two given cuastics are parametrized by an abelian surface; our surface $S$ is the quotient obtained by forgetting the order of the segments.

\smallskip

As we said, it would be natural to try to construct integral points on $\mP^3\setminus D$, where $D=Q_1\cup Q_2$,  by considering the family of fully integral bitangents to $D$. However, we have the following negative result:

\begin{thm}\label{few-fully-integral}
  Let $D\subset\mP^3$  be as above the divisor  $D:=Q_1\cup Q_2$. The set of bitangents to $D$ which are fully integral is degenerate. 
\end{thm}

\begin{proof}
Let us first consider the family of bisecants to $E$; denote this surface $\Xi$; it is obtained as the symetric square of $E$. In order that one such line be fully integral, it must neither reduce to one that is included in $Q_1$ or $Q_2$, nor reduce to a quadrisecant, ie.. an element of $\tilde{E}$. While the first condition is almost harmless, boiling down to an integrality condition with respect to a finite set in $\Xi$, the second one is an integrality with respect to a curve $\tilde{E}\subset \Xi$.

Now $\Xi$ admits a ruling over $E$ in the following way: choose a point $e\in E$ giving $E$ the structure of algebraic group; to every un-ordered pair $\{x,y\} \subset E$ we associate its sum $x+y\in E$ obtaining a morphism $\Xi\to E$. Its fibers are all isomorphic to $\mP^1$. The curve $\tilde{E}$ corresponds to a quadri-section. After the \`etale cover of the base $[2]: E\to E$ (multiplication by $2$-map), this quadrisection becomes the union of four disjoint rational sections of a fibration $\Xi'\to E$.

Now, by the already mentioned Chevalley-Weil theorem the set of integral points of $\Xi$ with respect to $\tilde{E}$ lifts (after finite extension of the field of definition) to a set of integral points on $\Xi'$ with respect to the four sections. But the complement of the four sections on $\xi'$ is isomorphic to the product $E\times \mP^1\setminus\{4\, \mathrm{points})$, so by Siegel's theorem is degenerate.

\smallskip

Let us consider now the other component of the family of bitangent. We denoted it by $\Sigma$; recall that it has sixteen singular points, corresponding to the sixteen bitangents completely contained in one of the two quadrics. It also contains a genus-one curve, isomorphic to $E$ and passing through these singular points, parametrizing the quadritangent lines. 

Recall that are looking at bitangent lines which do not reduce modulo any prime to lines contained in $Q_1\cup Q_2$ and do not reduce to quadritangent lines. Hence we are interested in rational points on $\Sigma$ which do not reduce modulo any prime neither to a singular point nor to the mentioned genus-one curve, which we denote again as $E$.

Now, under the birational isomorphism $S\to \Sigma$ the sixteen singular points correspond to sixteen $(-2)$ curves on the smooth surface $S$. Hence, the complement in $\Sigma$ of these singular points is isomorphic, as a quasi-projective surface, to the complement of these  curves on $S$. The genus-one curve $E$ on $\Sigma$ lifts to a curve on $S$ intersecting each of the sixteen mentioned $(-2)$-curves on $S$.

We claim that the integral points with respect to this seventeen component divisor in $S$ form a finite set. Indeed, recalling that $S$ is a Kummer surface, consider the degree-two  cover $\hat{A}\to S$, where $\hat{A}\to A$ is the blow-up of an abelian surface over its $2$-torsion points. By Chevalley-Weil theorem again, the rational   points on $S$ which are integral with respect to the sixteen mentioned $(-2)$-curves lift to rational points on $\hat{A}$ (after enlarging the number field in question). If these points are also integral with respect to $E_S$, they lift to points on $\hat{A}$ which are integral with respect to the pre-image of $E_S$ on $\hat{A}$. Pushing to $A$, we obtain rational points on $A$ which are integral with respect to a curve $E_A$, obtained as a double cover of a genus-one curve ramified over sixteen points. By Hurwitz formula, this curve $E_A$ has genus nine and so is (the support of) an ample divisor on $A$. By Faltings' theorem \cite{F}, the rational points on an abelian variety  which are integral with respect to an ample divisor are finite in number. 
\end{proof}

\smallskip

If we are just interested in the bitangent lines to $Q_1,Q_2$ which do not reduce to one of the sixteen bitangent lines contained in $Q_1,Q_2$, i.e. the integral - but not fully integral - bitangent lines,   we have the following

\begin{thm}
Under the Puncturing Conjecture, the set of integral bitangent lines is Zariski-dense in $S$.
\end{thm}

\begin{proof}
Consider the degree-two cover $\hat{A}\to S$ of the Kummer surface $S$, where again $\hat{A}$ is the blow-up of the abelian variety $A$ over the $2$-torsion.
Under the mentioned conjecture, the set of integral points of $\hat{A}$ with respect to the exceptional divisors is Zariski-dense. These integral points give rise to   points on $S$ which correspond to integral bitangent lines.
\end{proof}
\subsection{Other Fano threefolds}

We shall briefly discuss some other cases. 

\subsubsection{ The double solid case}
Let us consider the double solid branched over a smooth quartic surface and the related problem of producing integral points on the complement of such a surface. The variety of lines on the double solid is parametrized by an \'etale double cover of the variety of bitangent lines to the quartic surface. Now, in \cite{CoZucc-secondo} the authors proved that if the given quartic surface contains no line there are only finitely many rational   points on such a surface, which amounts to the existence of only finitely many bitangents which can be defined  over a given number field. Hence, the method of using such lines to produce integral points cannot work.

\subsubsection{ The hypercubic case} The case of a cubic hypersurface in $\mP^4$ can be handled to prove the density of integral points on the complement of one hyperplane sections. The reasoning is similar to the one used to prove the density of integral points on $\mP^3\setminus V$, where $V$ is a cubic surface.  The same holds, and is even easier, in higher dimensions.

\subsubsection{ The del Pezzo's threefold case}\label{casodidelpezzomenouno} Let us consider the case of the unique (up to birational isomorphism) Fano threefold   $B(5)$ obtained as a section of $G(2,5)\subset\mP^9$ by a linear subspace of codimension $3$. We prove that its integral points with respect to one hyperplane sections are Zariski-dense. Let $H$ be a hyperplane and $X^o=B(5)\setminus (H\cap B(5))$. Since the scheme of lines contained in $B(5)$ is isomorphic to $\mP^2$ and the set of lines in $H\cap B(5)$ is finite, we deduce from Lemma \ref{lemma-punture} the existense of a Zariski-dense set of fully integral lines. Hence the density of integral points on $X^o$. We will provide another proof of this fact in Proposition \ref{delpezzoretta}. We will treat also the case of two hyperplane sections but we are able to get the results only in some special cases, which, nevertheless are still normal crossing; see: Proposition \ref{delpezzoconica}.

\section{Blow-ups and integral points}
To understand $S$-integral points of a variety $X$ with respect to a divisor $D$ is about the log geometry of $X$; ultimately it is a problem which concerns the birational class of $X$. 

In this section we study some birational maneuvers on a $\mathbb K$-standard smooth rational fano threefold $X$ which enable to use the results of the previous section to study the (potential) density of integral points on $X$ with respect to a boundary $D$ such that $(X,D)$ is log-Fano or log-(Calabi-Yau).

\subsection{Smooth quadric: the log-Fano case}

\begin{prop}\label{dueiperpiani}
Let $Q_3\subset\mathbb P^4$ be a smooth quadric and let $H_1^{Q_3}$ and $H_2^{Q_3}$ be two hyperplane sections and $Q$ a smooth quadric section.
Then the subset of  the integral points with respect to $D=H_1^{Q_3}$, or $D=H_1^{Q_3}\cup H_2^{Q_3}$ or $D=Q$ is (potentially) dense.
\end{prop}
\begin{proof} The case of one or two hyperplane sections follows by Proposition \ref{treiperpiani}. Hence we treat only the case of the smooth quadric section.

Consider the projection from  a point $P\in Q$ and denote by  $T_P:=T_PQ_3\cap Q_3$ the hyperplane section given by the projective tangent space to $Q_3$ at the point $P$. This projection can be factorised in the following way:

\begin{equation}
\label{eq:(puntocubica)}
  \xymatrix{ & A \ar[dl]_{f} \ar@{=}[r] & A\ar[dr]^{f'} & \\
  Q_3\ar@{<-->}[rrr]  &  &  & \mathbb P^3,}
\end{equation}
where $f\colon A\to Q_3$ is the blow-up at $P$ and $f'\colon A\to \mathbb P^3$ is the contraction of the strict transform $H_P$ of $T_P$. Let $E:=f^{-1}(P)$. By standard theory $E\simeq\mathbb P^2$ and $\sO_E(E)\simeq \sO_{\mathbb P^2}(-1)$. On the other side we see that $f'(H_P)=C$ is a conic of $\mathbb P^3$ and letting $E'=f^{'-1}(C)$ we can write $E'=H_P$. Let $H$ and $L$ be the $f$-pull-back and the $f'$-pull-back of the hyperplane section of $Q_3$ and respectively of $\mathbb P^3$. Clearly $f'$ is given by the linear system $|H-E|$ while $f$ is given by $|2L-E'|$, that is $L=H-E$ and $H=2L-E'$.  We stress that $f'(E)$ is easily seen to be the plane $\Pi$ spanned by $C$. The strict transform $Q'$ of $Q$ belongs to $|2H-E|=|4L-2E'-E|=|3L-E'|$. That is $V=f'(Q')$ is trivially seen to be a smooth cubic containing the conic $C$. We consider the bijective morphism
\begin{equation}\label{chaivechiave}
g\colon Q_3\setminus ((T_PQ_3\cap Q_3)\cup Q)\to \mathbb P^3\setminus (\Pi\cup V)
\end{equation}

\noindent
induced by $f'\circ f^{-1}$. By Theorem \ref{cubicandplanewithacommonline} it follows that the set of rational points of $Q_3$ which are $((T_PQ_3\cap Q_3)\cup Q)$-integral is potentially dense. Then a fortiori those which are $Q$-integral are potentially dense.
\end{proof}

\subsection{Smooth quadric: the log-Calabi Yau case}
We recall that if $\mathbb G_m=\mP^1\setminus\{0,\infty\}$ then the subset of  the integral points is potentially dense, but we recall that if $\sO_S=\mathbb Z$ or the ring of integers of an imaginary quadratic field then $\mathbb G_m(\sO_S)$ is finite; see e.g.  \cite{C-libretto-verde}.

\begin{prop}\label{treiperpiani}
Let $Q_3\subset\mathbb P^4$ be a smooth quadric and let $H_1^{Q_3}$ and $H_2^{Q_3}$, $H_3^{Q_3}$ be three hyperplane sections.
Then the subset of  the integral points with respect to $H_1^{Q_3}\cup H_2^{Q_3}\cup H_3^{Q_3}$ is (potentially) dense.
\end{prop}
\begin{proof} By generality there exist two points $P,P'\in Q_3$ such that $H_1^{Q_3}\cap H_2^{Q_3}\cap H_3^{Q_3}=\{P,P'\}$. We consider the projection from $P\in H_1^{Q_3}$ and we use notations as in the proof of Proposition \ref{dueiperpiani}.
It remains defined an injective morphism
$$
g\colon Q_3\setminus ((T_PQ_3\cap Q_3)\cup H_1^{Q_3}\cup H_2^{Q_3}\cup H_3^{Q_3} )\to \mathbb P^3\setminus (\Pi_1\cup\Pi_2\cup\Pi_3\cup C)
$$
We denote by $\Pi_4\langle C\rangle$ the plane spanned by a conic. Since $\mathbb P^3\setminus (\Pi_1\cup\Pi_2\cup\Pi_3\cup \Pi_4)$ is $\mathbb G^3_m$, the claim follows.
\end{proof}

Now we show that Theorem \ref{two-quadrics-two concis} follows by Proposition \ref{treiperpiani} and it should be viewed as a typical example where birational geometry and arithmetic are strongly intertwingled. For reader's benefit we restate Theorem \ref{two-quadrics-two concis}:
\begin{thm}\label{complement by two quadrics special case}
Let $D=Q'_1\cup Q'_2$ where $Q'_1$ and $\Q'_2$ are two quadrics in $\mathbb P^3$ which contain a conic $C$. Then the set of the integral points of $\mathbb P^3$ with respect to $D$ is potentially dense.
\end{thm}
\begin{proof} We perform the inverse birational maneuver of the one done in the proof of Proposition \ref{dueiperpiani}. This time we denote by $f\colon A\to \mP^3$ the blow-up along $C$. We have that ${\rm{Pic}}(A)$ is free of rank $2$ and by standard two ray game theory we know that there exists a birational divisorial contraction $f'\colon A\to Q_3$ which contracts the $f$-strict transform $E'$ of the plane $\Pi_C$ containing $C$ to a point $P\in Q_3$. We set $E:=f^{-1}(C)$.
Let $H$ and $L$ be the $f$-pull-back and the $f'$-pull-back of the hyperplane section of $\mathbb P^3$ and respectively of $Q_3$. Clearly $f'\colon A\to Q_3$ is given by the linear system $|2H-E|$, that is $L=2H-E$ while $f\colon A\to \mathbb P^3$ is given by $|L-E'|$, that is $H=L-E'$. 
We denote by $Z_i$ the $f'$-image of the $f$-strict transform of $Q'_i$, where $i=1,2$. Moreover $f'(E)$ is the tangent hyperplane section $T_PQ_3\cap Q_3$, where $P=f'(E')$. It remains defined a bijective morphism
$$
g\colon \mathbb P^3\setminus (\Pi_C\cup D ) \to Q_3\setminus ((T_PQ_3\cap Q_3)\cup Z_1\cup Z_2)
$$
and by Equation \ref{chaivechiave} in the proof of Proposition \ref{dueiperpiani}, we conclude that the set of  integral points in the last affine variety is potentially dense. 
\end{proof}

\subsection{The complement of the quadric complex by a hyperplane section}
The quadric complex $X=Q_1\cap Q_2\subset\mathbb P^5$ is a smooth intersection of two quadrics $Q_1$, $Q_2$ in $\mathbb P^5$. It is a rational threefold of index $2$. According to Vojta's conjecture we expect that if $D\in |iH_{X}|$ where $i=1$ or $i=2$, and $H_{X}$ is an hyperplane section of $X$, then the integral points with respect to $D$ is potentially dense. 
\subsubsection{Smooth quadric complex: the log-Fano case}
We are able to solve the log-Fano case thanks to Proposition \ref{quadricaeduepiani}.
\begin{prop}\label{quadcompline}
Let $X=Q_1\cap Q_2\subset\mathbb P^5$ be a smooth quadric complex. Then the subset given by the integral points with respect to a smooth hyperplane section is potentially dense.
\end{prop}
\begin{proof} Fix a general hyperplane section $H_1^{X}\in |H_{X}|$. It is a del-Pezzo surface od degree $4$. We consider the projection from  a line $l\subset H_1^{X}$. It can be factorised in the following way:

\begin{equation}
\label{eq:(puntoquadriccomplex)}
  \xymatrix{ & A \ar[dl]_{f} \ar@{=}[r] & A\ar[dr]^{f'} & \\
  X\ar@{<-->}[rrr]  &  &  & \mathbb P^3,}
\end{equation}
where $f\colon A\to X$ is the blow-up at $l$ and $f'\colon A\to \mathbb P^3$ is the contraction of the strict transform $E'$ of the loci $D_l$ swept by lines in $X$ which intersect $l$.It is well-known that $D_l\in |2H^X|$. Let $H$ and $L$ be the $f$-pull-back and the $f'$-pull-back of the hyperplane section of $X$ and respectively of $\mathbb P^3$. Then $f$ is given by the linear system $|H-E|$. Following Tacheuchi's method see \cite[ Section 2 Projections of V from a point or a conic]{T} it can be shown that $-K_A=H+L$. Then since $-K_A$ is also linearly equivalent to $2H-E$ and to $4L-E'$ it follows that $L=H-E$ and that $H=3L-E'$. Moreover $f'$ contracts $E'$ onto a genus $2$ curve $C$ of degree $5$. In more geometrical terms it is easy to see that $E=\mathbb P(\sN_{l/X})$ is isomorphic to $\mathbb P^1\times\mathbb P^1$, that $f'_{|E}\colon E\to Q\subset \mathbb P^3$ is induced by the natural embedding of $\mathbb P^1\times\mathbb P^1$ into $\mathbb P^3$ and that $C\subset Q$ is an element of $|(2,3)|$. Moreover $E'\in |2H-3E|$.
Now since $l\subset H_1^{X}$ the $f'$-image of the $f$-strict transform of $H_1^{X}$ is a plane $\Pi$.

We consider the bjective morphism
$$
g\colon X\setminus (D_l\cup H_1^{X})\to \mathbb P^3\setminus (\Pi\cup Q)
$$
induced by $f'\circ f^{-1}$. 

By Proposition \ref{quadricaeduepiani} it follows that the set of points of $X$ which are $D_l\cup H_1^{X}$-integral is potentially dense. Then a fortiori those which are  $H_1^{X}$-integral is potentially dense.
\end{proof}

\subsubsection{Smooth quadric complex: a special case of the log-CY case}
We are able to solve the log-CY case thanks to Proposition \ref{quadricaeduepiani} in the case where the two hyperplane sections contain a common line.

\begin{prop}\label{quadriccomplexCY}
Let $X=Q_1\cap Q_2\subset\mathbb P^5$ be a smooth quadric complex. Let $D=H_1^{X}\cup H_2^{X}$ where $H_1^{X}$ and $H_2^{X}$ are two hyperplane sections such that there exists a line $l\subset H_1^{X}\cap H_2^{X}$. Then the subset given by the integral points with respect to $D$ is potentially dense.

\end{prop}
\begin{proof} As in the proof of Proposition \ref{quadcompline} we consider the projection from  the common line $l\subset H_1^{X}\cap H_2^{X}$ and by the same argument used there we can show that there exists a bijective morphism
$$
g\colon X\setminus (D_l\cup H_1^{X}\cup H_2^{X})\to \mathbb P^3\setminus (Q\cup\Pi_1\cup \Pi_2)
$$
where $\Pi_1$, $\Pi_2$ are respectively the $f'$-image of the $f$-strict transform of $H_1^{X}$, $H_2^{X}$ and $Q$ is the $f'$-image of the $f$-exceptional divisor. By  Proposition \ref{quadricaeduepiani} it follows that the set of rational points of $X$ which are $D_l\cup H_1^{X}\cup H_2^{X}$-integral is potentially dense. Then a fortiori those which are  $H_1^{X}\cup H_2^{X}$-integral is potentially dense.
\end{proof}
\begin{rem} We stress that the same method used in the proof of Proposition \ref{quadriccomplexCY} cannot be applied in the case of a general smooth hyperquadric section $Y$of $X$ and even in the case where the K3 $Y$ contains a line we are reduced to the study of integral points of the complement of $\mP^3$ by the union of a quadric and a quartic whose intersection contains a conic.
\end{rem}

\subsection{The del Pezzo Threefold}
The del Pezzo threefold $B_5\subset\mathbb P^6$ is another rational threefold. Vojta's conjecture predicts that if $D\in |iH_{B_5}|$ where $i=1$ or $i=2$, and $H_{B_5}$ is an hyperplane section of $B_5$, then the integral points with respect to $D$ is potentially dense. 
\subsubsection{A special case of the log-CY case for the del Pezzo Threefold}
We can treat with our method the case where the intersection of the two hyperplane sections contains a smooth conic.
\begin{prop}\label{delpezzoconica}
Let $B_5$ be a del Pezzo threefold. Let $D=H_1^{B_5}\cup H_2^{B_5}$ where $H_1^{B_5}$ and $H_2^{B_5}$ are two hyperplane sections which contain a smooth conic $q$. Then the subset given by the integral points with respect to $D$ is dense. 
\end{prop}
\begin{proof}
We denote by $H^{B_5}$ the hyperplane section of $B_5$ and we consider the projection from  a conic $q\subset H_1^{B_5}\cap H$. It can be factorised in the following way:

\begin{equation}
\label{eq:(puntodelpezzo)}
  \xymatrix{ & A \ar[dl]_{f} \ar@{=}[r] & A\ar[dr]^{f'} & \\
  B_5\ar@{<-->}[rrr]  &  &  & \mathbb P^3,}
\end{equation}
where $f\colon A\to B_5$ is the blow-up at $q$ and $f'\colon A\to \mathbb P^3$ is the contraction of the strict transform $E'$ of the loci $D_q$ swept by lines in $B_5$ which intersect $q$. By c.f. \cite[Proposition 2.1.3, Proposition 2.2.2]{TZ} we know that $E=\mathbb P(\sN_{l/X})$ is isomorphic to $\mathbb P^1\times\mathbb P^1$ and that, $D_p\in |2H^{B_5}|$. Let $H$ and $L$ be the $f$-pull-back and the $f'$-pull-back of the hyperplane section of $B_5$ and respectively of $\mathbb P^3$. Following Tacheuchi's method see \cite[ Section 2 Projections of V from a point or a conic]{T} it can be shown that $-K_A=H+L$. Then since $-K_A$ is also linearly equivalent to $2H-E$ and to $4L-E'$ it follows that $L=H-E$ and that $H=3L-E'$. Moreover $f'$ contracts $E'$ onto a rational curve $C$ of degree $4$ and by adjunction $f_{|E}\colon E\to Q$ is the natural embedding of $\mathbb P^1\times\mathbb P^1$ into $\mathbb P^3$. Finally $C\subset Q$ is an element of $|(1,3)|$ and $E'\in |2H-3E|$.
Now since $q \subset H_1^{X}\cap H^2_X$ the $f'$-image of the $f$-strict transform of  $H_1^{X}$ and $H_2^{X}$ is the plane $\Pi_1$ and respectively the plane $\Pi_2$, ( while the $f'$-image of the $f$-strict transform of the general hyperplane section is a smooth cubic surface containing the curve $C$).

We consider the bijective morphism
$$
g\colon X\setminus (D_q\cup H_1^{X}\cup H_2^{X})\to \mathbb P^3\setminus (Q\cup \Pi_1\cup\Pi_2)
$$
induced by $f'\circ f^{-1}$. By Proposition \ref{quadricaeduepiani} it follows that the set of $\mathbb K$ points of $X$ which are $D_q\cup H_1^{X}\cup H_2^{X}$-integral is dense. Then a fortiori those which are $H_1^{X}\cup H_2^{X}$-integral is dense.
\end{proof}
We point out the reader that Proposition \ref{delpezzoconica} obviously implies the log-Fano case. This case has been studied above in Subsection \ref{casodidelpezzomenouno}. Finally we provide here a third proof of the log-Fano case one to show the reader the differences between our various methods:
\begin{prop}\label{delpezzoretta}
Let $B_5$ be a del Pezzo threefold. Then the subset given by the integral points with respect to a hyperplane section is potentially dense.
\end{prop}
\begin{proof}
The hyperplane section $H^{B_5}$ is a del Pezzo surface of degree $5$. Consider the projection from a line $l\subset H_1^{B_5}$. It can be factorised in the following way:

\begin{equation}
\label{eq:(rettadelpezzo)}
  \xymatrix{ & A \ar[dl]_{f} \ar@{=}[r] & A\ar[dr]^{f'} & \\
  B_5\ar@{<-->}[rrr]  &  &  & \mathbb Q^3,}
\end{equation}
where $f\colon A\to B_5$ is the blow-up at $l$ and $f'\colon A\to \mathbb P^3$ is the contraction of the strict transform $E'$ of the loci $D_l$ swept by lines in $B_5$ which intersect $l$ to a (rational) twisted cubic $\Gamma$. Indeed by \cite[Sec. 2]{FN} and \cite[sec. 1]{Il}, c.f. see: \cite[Proposition 2.1.3]{TZ} we know that $E=\mathbb P(\sN_{l/X})$ is isomorphic to $\mathbb F_1$ and that, $D_l\in |H^{B_5}|$. Let $H$ and $L$ be the $f$-pull-back and the $f'$-pull-back of the hyperplane section of $B_5$ and respectively of $\mathbb P^3$. Following Tacheuchi's method see \cite[ Section 2 Projections of V from a point or a conic]{T}, it can be shown that $-K_A=H+L$. Then since $-K_A$ is also linearly equivalent to $2H-E$ and to $3L-E'$ it follows that $L=H-E$ and that $H=2L-E'$.  Since $l \subset H_1^{B_5}$ the $f'$-image of the $f$-strict transform of $H_1^{B_5}$ is an hyperplane $H_1^{Q_3}$ of $Q_3$. It is now easy to show that $f'$ contracts $E'$ onto $\Gamma$ and by adjunction $f'_{|E}\colon E\to Q_3\subset\mathbb P^4$ is the natural embedding of $\mathbb P^1\times\mathbb P^1$ into the hyperplane $\mP^3$ spanned by $\Gamma$. We set $f'_{|E}(E)=H_2^{Q_3}$. Finally $\Gamma\subset Q_3$ is an element of $|(1,3)|$ in $H_2^{Q_3}$; we stress that $H_2^{Q_3}\subset\mathbb P^3\cap Q_3\subset\mathbb P^4$. We consider the bijective morphism
$$
g\colon X\setminus (D_l\cup H_1^{B_5})\to Q_3\setminus (H_1^{Q_3}\cup H_2^{Q_3} )
$$
induced by $f'\circ f^{-1}$. By  Proposition \ref{treiperpiani} it follows that the set of rational points of $B_5$ which are $D_l\cup H_1^{B_5}$-integral is dense. Then a fortiori those which are $H_1^{B_5}$-integral is dense.
\end{proof}
\begin{rem} Although the statement of Proposition \ref{delpezzoretta} has perhaps even simpler proofs, its proof shows how through a chain of suitable birational moves it is possible to obtain results on the density of integer points.
\end{rem}

\subsection{The complement of the singular cubic threefold by two hyperplanes}
 The following case is a bit different since it deals with a singular rational variety. Moreover we need to assume that both the two hyperplane sections passes through the singular point.

\begin{prop}\label{cubic threefold}
Let $B(3)$ be a cubic threefold with a O.D.P. Let $D=H_1^{B(3)}\cup H_2^{B(3)}$ where $H_1^{B(3)}$, $H_2^{B(3)}$ are two hyperplane sections passing through the singular point. Then the subset given by the integral points with respect to $D$ is potentially dense.
\end{prop}
\begin{proof} We consider the point $P\in\mathbb P^4$ where B(3) is singular. Let $\pi\colon \mathbb P=\mathbb P(\sO_{\mathbb P^3}\oplus\sO_{\mathbb P^3}(1))\to \mathbb P^4$ be the blow-up of $P\in\mathbb P^4$ at $P$. We denote by $A$ the strict transform of $B(3)$. Let $\pi'\colon \mathbb P\to\mathbb P^3$ be the natural projection and let $H$, $L$ be the pull-back of the hyperplane section of respectively $\mathbb P^4$, $\mathbb P^3$. Note that $H$ is a section of the tautological line bundle $\sO_\mathbb P(1)$. 
Let $w_{\infty}\in H^0(\mathbb P,\sO_\mathbb P(1)\otimes_\mathbb P\sO_\mathbb P(-L))$. It is an easy check to verify that $A$ is a smooth element of $|H+2L|$. More precisely if 
$\mathbb P^4={\rm{Proj}}(\mathbb C[u_0,u_1,u_2,u_3,u_4])$, $\mathbb P^3={\rm{Proj}}(\mathbb C[x_0,x_1,x_2,x_3])$, $P=[0:0:0:0:1])$ then 
$$\pi\colon ([w_\infty,w],[x_0,x_1,x_2,x_3])\mapsto [w_{\infty} x_0,w_{\infty} x_1, w_{\infty} x_2, w_{\infty}x_3,w]=[u_0,u_1,u_2,u_3,u_4]$$ while $\pi' \colon ([w_\infty,w],[x_0,x_1,x_2,x_3])\mapsto [x_0,x_1,x_2,x_3]$. In particular we can assume $B(3)=(u_4f_2(u_0,u_1,u_2,u_3)+f_3(u_0,u_1,u_2,u_3)=0)$ where $f_2$, $f_3$ are general homogeneous polynomial in the variables $u_0,u_1,u_2,u_3$ o degree $2$  and respectively $3$ and  
$$
A=wf_2(x_0,x_1,x_2,x_3)+w_{\infty} f_3(x_0,x_1,x_2,x_3)=0
$$
where $w\in H^0(\mathbb P,\sO_\mathbb P(1))$.
We consider the following diagram
\begin{equation}
\label{eq:(puntocubicasingolae)}
  \xymatrix{ & A \ar[dl]_{f} \ar@{=}[r] & A\ar[dr]^{f'} & \\
  B(3)\ar@{<-->}[rrr]  &  &  & \mathbb P^3,}
\end{equation}
where $f\colon A\to B(3)$ is the blow-up of $B(3)$ at $P$ (and it is induced by $\pi$) while $f'\colon A\to \mathbb P^3$ is the restriction of $\pi'$ to $A$. By adjunction it holds that $-K_A=H+L$ where, by abuse of notation, we still denote by $H$, $L$ the pull-back of the hyperplane sections of respectively $B(3)$ and $\mathbb P^3$. Let $E:=(w_\infty=0)_{|A}$ be the exceptional divisor of $f$, which is easily seen to be a quadric and let $E'$ be $\pi^{'-1}(C)$ where $C\subset\mP^3$ is the smooth genus-$4$ curve given by $ f_2(x_0,x_1,x_2,x_3)= f_3(x_0,x_1,x_2,x_3)=0$. Clearly $E'$ is the unique divisor on $A$ which is contracted by $f'$. Moreover $f'$ is the blow-up at $C$. A simple computation shows that $L=H-E$ and $H=3L-E'$. In particular $H_1^{Q_3}$ and and $H_2^{Q_3}$  are transformed into the planes $\Pi_1$ and respectively $\Pi_2$ while $f'(E)=Q$ is obviously the quadric $Q:=(f_2(x_0,x_1,x_2,x_3)=0)\subset\mathbb P^3$. We set $T:=f(E')$. We have the bijective morphism
$$
g\colon B(3)\setminus (T\cup H_1^{B(3)}\cup H_2^{B(3)})\to \mathbb P^3\setminus (\Pi_1\cup\Pi_2\cup Q)
$$
induced by $f'\circ f^{-1}$. By Proposition \ref{quadricaeduepiani} it follows that the set of rational points of $B(3)$ which are $(T\cup H_1^{Q_3}\cup Q_2^{Q_3})$-integral is potentially dense. Then a fortiori those which are $H_1^{Q_3}\cup Q_2^{Q_3}$-integral is potentially dense.
\end{proof}

\begin{ackn}
This research is supported by PRIN 2017 - Geometric, Algebraic and Analytic Aspects of Arithmetics. \end{ackn}

\end{document}